\documentclass[10pt]{article}

\usepackage[utf8]{inputenc}  % se compili con pdflatex
\usepackage[T1]{fontenc}
\usepackage{lmodern}
\usepackage[english]{babel}  % o italian se vuoi sillabazione italiana
\usepackage[a4paper,margin=1in]{geometry}
\usepackage{microtype}
\usepackage{graphicx}
\usepackage{booktabs}        % tabelle belle

\usepackage{amsmath,amsfonts,amssymb,amsthm}

\newcommand{\F}{\mathbb{F}}

\newtheorem{theorem}{Theorem}[section]

\newtheorem{definition}[theorem]{Definition}

\newtheorem{oss}[theorem]{Remark}
\usepackage{siunitx}
\usepackage{enumitem}        % (una sola volta)
\usepackage{xcolor}
\usepackage{csquotes}
\usepackage{subcaption}
\usepackage{placeins}
\usepackage{pdflscape}
\usepackage{rotating}

\usepackage{authblk}

\usepackage[hidelinks]{hyperref}
\hypersetup{
	colorlinks=true,
	linkcolor=blue!60!black,
	citecolor=blue!60!black,
	urlcolor=blue!60!black,
	pdfauthor={Isabella Mastroianni},
	pdftitle={Graph-Based Retrieval Strategy for RAG Systems in Hyperlinked Technical Documentation}
}
\usepackage{orcidlink} % dopo hyperref

\title{Hook-decomposable modules and their resolutions}

\author[1,2]{{Isabella Mastroianni\orcidlink{0009-0002-9866-3648}}}
\author[2]{Marco Guerra\orcidlink{0000-0002-1825-0097}}
\author[2]{Ulderico Fugacci\orcidlink{0000-0003-3062-997X}}
\author[1]{Emanuela De Negri\orcidlink{0000-0002-1825-0097}}

\affil[1]{Department of Mathematics (DIMA), University of Genova, Italy}
\affil[2]{Institute of Applied Mathematics and Information Technologies “Enrico Magenes” (IMATI), National Research Council, Italy}

\date{March 2026}

\begin{document}

\maketitle

\begin{abstract}
We compare several classes of biparameter persistence modules: $\gamma$-products of monoparameter modules, hook-decomposable modules, modules admitting a Smith-type structure theorem, and modules of projective dimension at most 1. We determine all logical implications among these classes, providing explicit counterexamples showing that the converses fail when appropriate. In particular, $\gamma$-products (i.e., hook-decomposable modules) form a very small subclass of biparameter modules, precisely the ones for which a structure theorem still holds, thus making explicit the richer structural complexity of the biparameter setting compared to the monoparameter one.
\end{abstract}

\section{Motivation and contribution}
\label{sec:objective}

Persistent homology is a fundamental tool in topological data analysis, with monoparameter persistence extensively studied and applied \cite{edelsbrunner2010computational, carlsson2021topological, carlsson2009topology, DONUT}. 
Multiparameter persistence, while offering greater discriminative power \cite{carlsson2007theory, kerber2020multi}, remains less commonly used, partly due to the lack of a simple structural interpretation of its capabilities. A natural way to investigate this difference could be to examine biparameter modules obtained from constructions on monoparameter ones, in order to understand how such “simple” biparameter modules sit inside the broader category.

In previous work \cite{mastroianni2025retrieving}, we introduced one such construction, the $\gamma$-product, and showed that the resulting biparameter modules coincide precisely with hook-decomposable modules. Motivated by this, in the present work we analyze the logical relations between four classes of biparameter persistence modules naturally arising from, or closely related to, this construction: $(i)$ products of monoparameter modules (i.e.,\ $\gamma$-products), $(ii)$ hook-decomposable modules, $(iii)$ modules for which a Smith-type structure theorem applies, $(iv)$ modules of projective dimension~1, $(v)$ free modules.
We show that $(i)$, $(ii)$, and $(iii)$ are equivalent, and that $(v)$ implies $(ii)$, which in turn implies $(iv)$; none of these last two implications has a converse. The resulting diagram of implications is given in Figure~\ref{fig:implications}, while a schematic arrangement of these classes, as it will follow from the implications proved below, is shown in Figure~\ref{fig:classifications}.

\begin{figure}
    \centering
    \includegraphics[width=\linewidth]{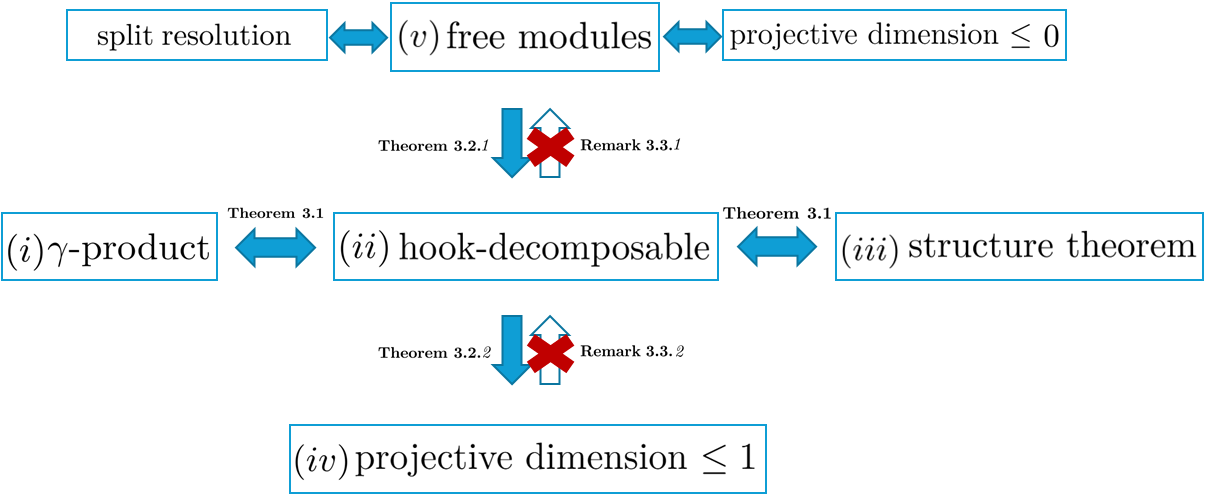}
    \caption{Implication diagram that we establish in this work.}
    \label{fig:implications}
\end{figure}

\begin{figure}
    \centering
    \includegraphics[width=0.85\linewidth]{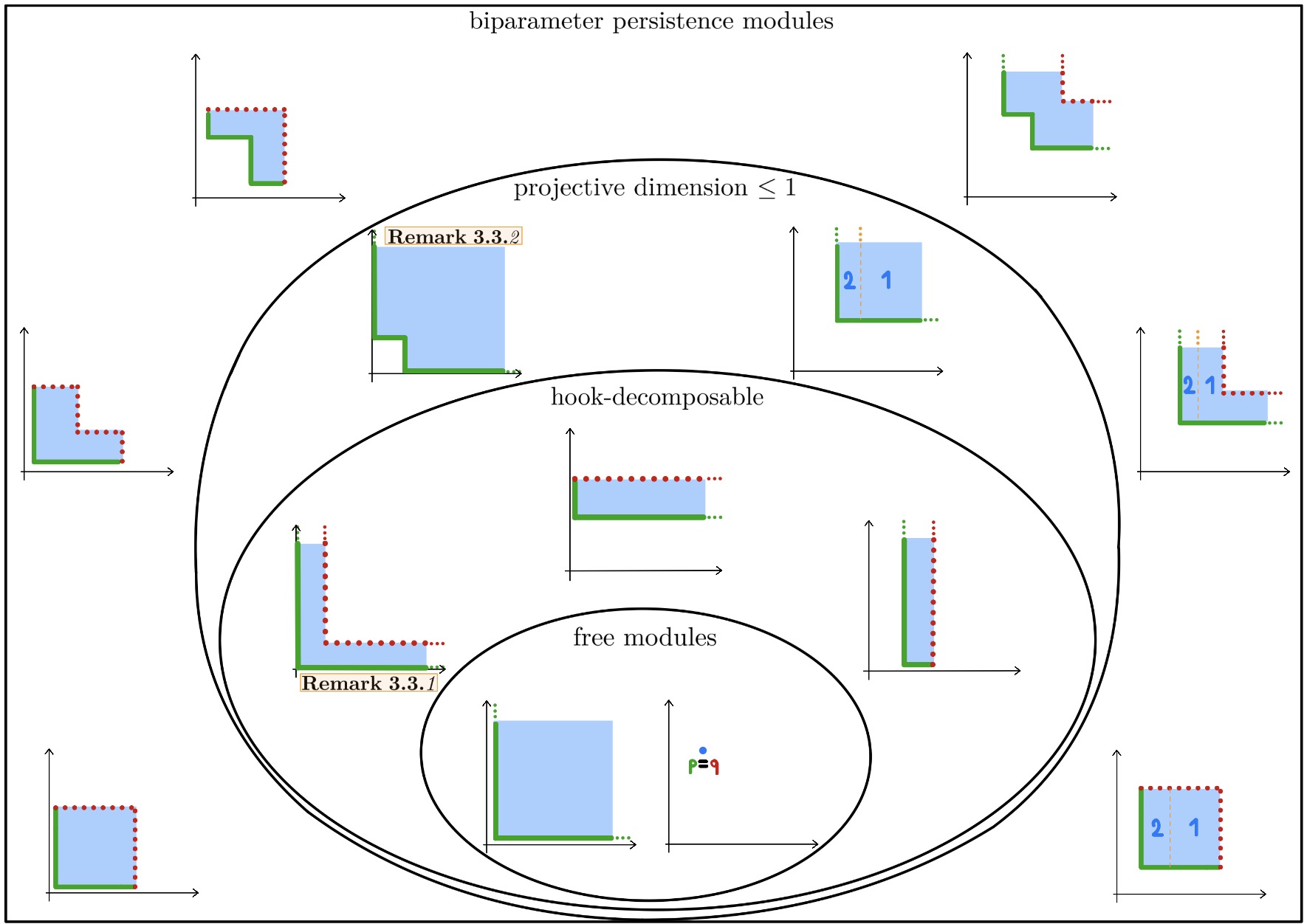}
    \caption{Visual representation of the studied classes. Examples of possible supports are shown for each one. \scriptsize Lines: green = births; red-dotted = deaths; orange-dashed = merges. Blue areas = supports interior.}
    \label{fig:classifications}
\end{figure}

\section{Background notions}

Let $\F$ be a field and \(M\) be a biparameter persistence module, i.e., a finitely generated, standard bigraded $\F[x,y]$-module. We recall the notions relevant to our analysis.

\begin{definition}\label{def}
Let $M$ be as above.
\begin{enumerate}
    \item \emph{$\gamma$-product.}  
    A construction producing a biparameter module from two monoparameter modules.  
    (We refer to \cite{mastroianni2025retrieving} for the precise definition.)

    \item \emph{Hook-decomposable modules.}  
    We call \(M\) \emph{hook-decomposable} if it is a direct sum of hook submodules, i.e.,\ interval modules \(I\) for which there exist bigrades \(\mathbf{p} \le \mathbf{q}\), with \(\mathbf{q}\) possibly having infinite coordinates, such that its support\footnote{The support of a (multi)graded module is the set of (multi)degrees in which the module is nonzero.
    } is of the form \(Supp(I)=\{\boldsymbol{\alpha}\in\mathbb{N}^2\thinspace |\thinspace \mathbf{p} \le \boldsymbol{\alpha}, \boldsymbol{\alpha} \not\ge \mathbf{q} \}.\)  
    Such class of modules is relevant due to the fact that the rank exact resolution of any module always consists of hook-decomposable modules \cite{botnan2024bottleneck, bjerkevik2025stabilizing,botnan2025signed}.

    \item \emph{Structure theorem.}
    We say that $M$ \emph{satisfies the structure theorem} if it decomposes as
    \[
        M \;\cong\;
        \bigoplus_{k=1}^f \F[x,y](-\mathbf{p}_k)
        \;\oplus\;
        \bigoplus_{j=1}^t
            \big(\F[x,y] \big/ (\mathbf{x}^{\,\mathbf{q}_j-\mathbf{p}_j})\big)(-\mathbf{p}_j),
    \]\label{eq:struct thm}
    for some bigrades $\mathbf{p}_k,\mathbf{p}_j,\mathbf{q}_j$, where \(\mathbf{x}^{\,\mathbf{q}_j - \mathbf{p}_j}=x^{q_{jx}-p_{jx}}y^{q_{jy}-p_{jy}}\), with \(\mathbf{p}_j=(p_{jx},p_{jy}), \mathbf{q}_j=(q_{jx},q_{jy})\).
    
    Equivalently, $M$ admits a Smith diagonal form. 
    
    \item \emph{Resolution.}
    A \emph{(minimal) graded free resolution} of \(M\) has the form
    \[
        0 \longrightarrow F^{\beta_2} \longrightarrow F^{\beta_1} \longrightarrow F^{\beta_0} 
        \longrightarrow M \longrightarrow 0,
    \]
    where each free module is
    \(
        F^{\beta_i} \;:=\; 
        \bigoplus_{j} \F[x,y]^{\beta_{i,j}}(-\mathbf{p}_{i,j}),
    \)
    the integers \(\beta_{i,j}\) are the graded Betti numbers of \(M\),
    and 
    \(
        \beta_i \;:=\; \sum_j \beta_{i,j}.
    \)
    Each \(\mathbf{p}_{i,j}\) denotes the bidegree of the \(j\)-th generator in \(F^{\beta_i}\).

    \item \emph{Projective dimension.}  
    The module \(M\) has \emph{projective dimension} \(\leq 1\) if and only if \(\beta_2=0\) in its minimal free resolution.

    %if $\beta_2 = 0$ we say that $M$ has \emph{projective dimension~1} if and only if $\beta_1 \neq 0$, otherwise $M$ has \emph{projective dimension~0}.
    
    % \item \emph{Split exact sequence.}  
    % A short exact sequence 
    % \(
    % 0 \to A \to B \to C \to 0
    % \)
    % \emph{splits} if \(B \cong A \oplus C\), equivalently if the injection admits a retraction or the surjection admits a section.
\end{enumerate}
\end{definition}

    \begin{oss}
        Please note that freeness (condition \((v)\)) is equivalent to projective dimension \(\leq 0\) and to the minimal free graded resolution being split short exact.
    \end{oss}

\section{Hook-decomposability: sufficient and necessary conditions}

\begin{theorem}\label{thm: hook-charact}\emph{(Characterization of hook-decomposability)}
    The following are equivalent: \((i)\,M\) is a \(\gamma\)-product; \((ii)\,M\) is hook-decomposable; \((iii)\) the structure theorem holds for \(M\). 
\end{theorem}
    
\begin{proof}[Sketch of the proof]
    The equivalence between \((i)\) and \((ii)\) was established in \cite{mastroianni2025retrieving}.
    
    For the equivalence between \((ii)\) and \((iii)\), observe that a hook summand is by definition of the form \(\F[x,y](-\mathbf{p})\) or \(\big(\F[x,y] \big/ (\mathbf{x}^{\,\mathbf{q}-\mathbf{p}})\big)(-\mathbf{p})\), where \(\mathbf{p} \le \mathbf{q}\) are the bigrades determining its support as in Definition~\ref{def} (\(\mathbf{q}\) has infinite coordinates in the first case). % \textcolor{red}{bigradi in bold}. 
    Thus, \(M\) is hook-decomposable if and only if it decomposes as in~\eqref{eq:struct thm}, that is, if and only if the structure theorem holds for \(M\).
\end{proof}

\begin{theorem}\emph{(Properties of the resolutions of hook-decomposable modules)}
    \begin{enumerate}
        \item \((v) \Rightarrow (ii)\). If \(M\) is free, then \(M\) is hook-decomposable.
        \item \((ii) \Rightarrow (iv)\). If \(M\) is hook-decomposable, then \(M\) has projective dimension \(\leq 1\). 
    \end{enumerate}
        
\end{theorem}

\begin{proof}[Sketch of the proof]
\begin{enumerate}
    % \item From the hypothesis it follows that \(M\) is composed by null or free summands, hence hook summands.
    Each free summand of \(M\) is either null or has a relation at infinity.
    \item From Theorem \ref{thm: hook-charact}, it follows that each hook summand is of the form \(\F[x,y](-\mathbf{p})\) or\\ \(\big(\F[x,y] \big/ (\mathbf{x}^{\,\mathbf{q}-\mathbf{p}})\big)(-\mathbf{p})\), so their resolutions are \(0 \rightarrow \F[x,y](-\mathbf{p}) \rightarrow \F[x,y](-\mathbf{p}) \rightarrow 0\) and \(0 \rightarrow \F[x,y](-\mathbf{q}) 
        \xrightarrow{\;\mathbf{x}^{\,\mathbf{q}-\mathbf{p}}\;} 
         \F[x,y](-\mathbf{p}) 
         \rightarrow \big(\F[x,y] \big/ (\mathbf{x}^{\,\mathbf{q}-\mathbf{p}})\big)(-\mathbf{p})
         \rightarrow 0,\) respectively.
\end{enumerate}
\end{proof}

\begin{oss}
    We now exhibit two examples, included in Figure \ref{fig:classifications}, that show that the converse implications do not hold.
\begin{enumerate}
    \item The module with one generator \(g\) in bidegree \((0,0)\) and the relation \(xyg\) in bidegree \((1,1)\) is hook-decomposable yet its not free. Notably, its resolution does not split, as the injection admits no graded retraction%, yet its resolution does not split, since the injection admits no linear retraction
    .
    \item The module with generators \(g\) in bidegree $(0,1)$ and \(h\) in bidegree $(1,0)$ and the relation \(xg-yh\) in bidegree $(1,1)$ has projective dimension~1, yet it is not hook-decomposable.
\end{enumerate}
\end{oss}

\section*{Acknowledgements}
This work was carried out within the framework of the following projects:
the \textit{Programma Regionale Fondo Sociale Europeo+ 2021--2027, Priorità~2 -- Istruzione e Formazione -- ESO~4.6 (OS-f)}, with the contribution of \textit{Rulex Innovation Labs S.r.l.};
the \textit{National Recovery and Resilience Plan (NRRP)}, Mission~4, Component~2, Investment~1.5, by the \textit{European Union -- NextGenerationEU} and the \textit{Ministry of University and Research (MUR)};
the \textit{National Center for HPC, Big Data and Quantum Computing} (CN0000013);
Italian \textit{National Biodiversity Future Center (NBFC)} - \textit{National Recovery and Resilience Plan (NRRP)} funded by the \textit{European Union - NextGenerationEU} (project code CN 00000033);
\textit{Innovation Ecosystem Robotics and AI for Socio-economic Empowerment (RAISE)} -- \textit{National Recovery and Resilience Plan (NRRP)} funded by the \textit{European Union -- NextGenerationEU} (project code ECS 00000035).

free modules

\bibliographystyle{plainurl}
\bibliography{refs}

@book{edelsbrunner2010computational,
  title={Computational topology: an introduction},
  author={Edelsbrunner, Herbert and Harer, John L},
  year={2010},
  publisher={American Mathematical Society}
}

@book{carlsson2021topological,
  title={Topological data analysis with applications},
  author={Carlsson, Gunnar and Vejdemo-Johansson, Mikael},
  year={2021},
  publisher={Cambridge University Press}
}

@article{carlsson2009topology,
  title={Topology and data},
  author={Carlsson, Gunnar},
  journal={Bulletin of the American Mathematical Society},
  volume={46},
  number={2},
  pages={255--308},
  year={2009}
}

@inproceedings{carlsson2007theory,
  title={The theory of multidimensional persistence},
  author={Carlsson, Gunnar and Zomorodian, Afra},
  booktitle={Proceedings of the twenty-third annual symposium on Computational geometry},
  pages={184--193},
  year={2007}
}

@inproceedings{kerber2020multi,
  title={Multi-parameter persistent homology is practical},
  author={Kerber, Michael},
  booktitle={TDA \& Beyond},
  year={2020}
}

@article{botnan2024bottleneck,
  title={On the bottleneck stability of rank decompositions of multi-parameter persistence modules},
  author={Botnan, Magnus Bakke and Oppermann, Steffen and Oudot, Steve and Scoccola, Luis},
  journal={Advances in Mathematics},
  volume={451},
  pages={109780},
  year={2024},
  publisher={Elsevier}
}

@article{bjerkevik2025stabilizing,
  title={Stabilizing decomposition of multiparameter persistence modules},
  author={Bjerkevik, H{\aa}vard Bakke},
  journal={Foundations of Computational Mathematics},
  pages={1--60},
  year={2025},
  publisher={Springer}
}

@article{botnan2025signed,
  title={Signed barcodes for multi-parameter persistence via rank decompositions and rank-exact resolutions},
  author={Botnan, Magnus Bakke and Oppermann, Steffen and Oudot, Steve},
  journal={Foundations of Computational Mathematics},
  volume={25},
  number={5},
  pages={1815--1874},
  year={2025},
  publisher={Springer}
}

@misc{DONUT,
  author = {Giunti, Barbara and Lazovskis, J{\=a}nis and Rieck, Bastian},
  title  = {
    {DONUT}: {D}atabase of {O}riginal \& {N}on-{T}heoretical {U}ses of {T}opology
  },
  note   = {\url{https://donut.topology.rocks}},
  year   = {2022},
  key    = {DONUT},
}

@article{mastroianni2025retrieving,
  title={Retrieving biparameter persistence modules from monoparameter ones: a characterization of hook-decomposable persistence modules},
  author={Mastroianni, Isabella and Guerra, Marco and Fugacci, Ulderico and De Negri, Emanuela},
  journal={arXiv preprint arXiv:2506.14678},
  year={2025}
}

\end{document}